\documentclass[11pt]{article}
\usepackage{arxiv}
\usepackage{amsfonts}
\usepackage{amsmath}
\usepackage{amsthm}
\usepackage{pxfonts}
\usepackage{mathtools}
\usepackage{tocloft}
\usepackage{todonotes}
\usepackage{extarrows}
\usepackage{titlesec}
\usepackage{enumitem}

\usepackage{hyperref}

\theoremstyle{definition}
\newtheorem{definition}{Definition}[]

\newtheorem{remark}[definition]{Remark}

\theoremstyle{plain}
\newtheorem{theorem}[definition]{Theorem}
\newtheorem{corollary}[definition]{Corollary}
\newtheorem{proposition}[definition]{Proposition}
\newtheorem{lemma}[definition]{Lemma}

\newcommand{\norm}[1]{\left\lVert#1\right\rVert}

\begin{document}

\title{$L^p(I,C^\alpha(\Omega))$ Regularity for Reaction-Diffusion Equations with Non-smooth Data}

\date{\today}
\author{
    Patrick Dondl \\
    Department of Applied Mathematics\\
    University of Freiburg \\
    Hermann-Herder-Stra\ss e 10, 79104 Freiburg i. Br., Germany \\
    \texttt{patrick.dondl@mathematik.uni-freiburg.de}
   \And
    Marius Zeinhofer \\
    Department of Applied Mathematics\\
    University of Freiburg\\
    Hermann-Herder-Stra\ss e 10, 79104 Freiburg i. Br., Germany \\
    \texttt{marius.zeinhofer@mathematik.uni-freiburg.de}
}

\maketitle

\begin{abstract}
    We prove an $L^p(I,C^\alpha(\Omega))$ regularity result for a reaction-diffusion equation with mixed boundary conditions, symmetric $L^\infty$ coefficients and an $L^\infty$ initial condition. We provide explicit control of the $L^p(I,C^\alpha(\Omega))$ norm with respect to the data. To prove our result, we first establish  $C^\alpha(\Omega)$ control of the stationary equation, extending a result by \cite{haller2009holder}.
\end{abstract}


\section{Introduction}
In this article we are interested in the $L^p(I,C^\alpha(\Omega))$ regularity of the solution $v$ to a reaction-diffusion equation of the form
\begin{align*}
    d_tv - \operatorname{div}(D\nabla v) + v &= f, \quad \ \ \text{in }I\times\Omega
    \\
    v(0) &= v_0, \quad  \text{in }\Omega
    \\
    v(t,x) &= 0, \quad \ \ t\in I, x\in\Gamma_D,
    \\
    n\cdot D(x)\nabla v(t,x) &= 0, \quad \ \  t\in I, x\in\Gamma_N.
\end{align*}
where $\Omega\subset \mathbb{R}^d$, $d=2,3$ is a bounded Lipschitz domain, $\partial\Omega = \Gamma_D\cup \Gamma_N$ is a partition of the boundary in Dirichlet and Neumann part, $I=[0,T]$ is a finite time interval, and $f\in L^p(I,L^2(\Omega))$. In particular, we consider the case where the initial condition $v_0$ only has regularity $v_0\in L^\infty(\Omega)$ and we make few assumptions on the remaining data. Our main result is the following.
\begin{theorem}\label{theorem:parabolic_quantitative_holder}
    Let $\Omega \subset \mathbb{R}^d$ with $d=2,3$ be a Lipschitz domain, $I=[0,T]$ a time interval, $\partial\Omega = \Gamma_N\cup\Gamma_D$ a partition of the boundary into a Dirichlet and a Neumann part, where both $\Gamma_N$ and $\Gamma_D$ are allowed to have vanishing measure. Assume that $\Omega\cup\Gamma_N$ is Gr\"oger regular, let $f\in L^p(I,L^2(\Omega))$ for $p\in[2,\infty)$, $D\in L^\infty(\Omega,\mathbb{R}^d)$ be symmetric and elliptic with ellipticity constant $\nu > 0$. For $v_0\in L^\infty(\Omega)$ denote by $v\in H^1(I,H^1_D(\Omega),H^1_D(\Omega)^*)$ the solution to 
    \begin{align*}
        \int_I \langle d_tv,\cdot\rangle_{H^1_D(\Omega)}\mathrm dt + \int_I\int_\Omega D\nabla v \nabla \cdot + v(\cdot)\mathrm dx\mathrm dt
        &=
        \int_I\int_\Omega f(\cdot)\mathrm dx\mathrm dt \quad \text{in }L^2(I,H^1_D(\Omega))^*
        \\
        v(0) &= v_0.
    \end{align*}
    Then there is $\beta = \beta(p)\in (0,1)$ such that $v\in L^p(I,C^\beta(\Omega))$ and we may estimate
    \begin{equation}\label{equation:parabolic_quantitative_holder}
        \norm{v}_{L^p(I,C^\beta(\Omega))} \leq C\left(\Omega, T, \nu, \lVert D \rVert_{L^\infty(\Omega,\mathbb{R}^{d\times d})}, p, \beta \right)\cdot \left[ \lVert f \rVert_{L^p(I,L^2(\Omega))} + \lVert v_0 \rVert_{L^\infty(\Omega)} \right].
     \end{equation}
     In the above estimate, if we fix $\Omega$ and $p$, only a lower bound for $\nu$ and upper bounds for $\norm{D}$ and $T$ determine the value of the constant $C$. The regularity estimate is thus  uniform for $\nu\in[c_E,C_E]$, $D\in L^\infty(\Omega, \mathcal{M}_s)$ with $\norm{D}\leq C_B$ and time intervals $I^*=[0,T^*]$ with $T^* \leq T$. 
\end{theorem}
The crucial detail in the above theorem is the fact that $v_0$ lies only in the space $L^\infty(\Omega)$ and not in the trace space for the initial conditions. Therefore, well known maximal regularity results, for example \cite{amann1995linear}, cannot be applied directly. We thus split the problem into two equations, one with homogeneous right-hand side and one with homogeneous initial condition and analyze them separately. We remark that we are only concerned with  spatial dimensions two and three and that our proof does not extend beyond this. The reason lies in the stationary counterpart of the result, Theorem 5.1 in \cite{haller2009holder}, where this restriction on the dimension appears. 

There are a number of reasons to study regularity properties of equations with non-smooth data. Often, mixed boundary conditions are dictated by concrete applications and this alone leads to a considerable loss of regularity, at least if regularity up to the boundary is needed, see \cite{savare1997regularity, kassmann2004regularity}. Another reason to study problems with non-smooth data comes from multi-physics problems, i.e., coupled systems of equations. To prove existence results for coupled systems one usually employs an ansatz based on a fixed-point theorem and successively solves the equations. This leads to problems with low regularity data as it may be necessary to frame the fixed point problem in a low regularity setting. Providing explicit norm control in the sense of Theorem~\ref{theorem:parabolic_quantitative_holder} is useful, e.g., for (PDE constrained) optimization problems where the existence of a solution is established by the direct method of the calculus of variations and thus bounds on the solution independent of the data are required. In forthcoming work we discuss how our main result is crucial in establishing the existence of an optimal control function in a PDE constrained optimization problem stemming from tissue engineering.

The proof of Theorem~\ref{equation:parabolic_quantitative_holder} crucially relies on a counterpart for the stationary problem. The stationary result is available in the literature, see Theorem 5.1 in \cite{haller2009holder}, albeit without the information on the norm control. We therefore revisit the proof of Theorem 5.1 in \cite{haller2009holder} and provide the missing estimates required for explicit norm control. For the reasons given above, this is of independent interest, so we provide Theorem~\ref{thm:quantitative_holder_regularity} for a precise statement of our quantitative version of Theorem 5.1 in \cite{haller2009holder}. 

There are several results in the literature that treat the regularity of elliptic and parabolic equations subject to Dirichlet-Neumann conditions for non-smooth domains and rough coefficients. We briefly discuss the ones most closely related to our contribution. The mere H\"older regularity of elliptic equations for mixed boundary conditions was already established in \cite{haller2009holder}, however without explicit control of the H\"older norm and the implications for parabolic problems where recognized in \cite{disser2017holder}. There, a $C^\alpha(I\times\Omega)$ regularity result for a diffusion equation (without reaction term) was provided, however, the authors neither consider $L^\infty(\Omega)$ initial conditions (under which the $C^\alpha(I\times\Omega)$ regularity can not hold in general) nor do they provide explicit control of the H\"older norm in terms of the data. Other works focus on the maximal regularity of parabolic equations in distribution spaces, see for instance \cite{haller2011maximal} or maximal regularity questions for non-autonomous equations, see \cite{disser2017holder}.

\subsection{Preliminaries and Notation}
Let $\Omega\subset\mathbb{R}^d$, for $p\in[1,\infty]$ and $k\in\mathbb{N}$ we denote by $L^p(\Omega)$ the space of $p$-integrable functions, by $W^{k,p}(\Omega)$ the subset of $L^p(\Omega)$ of $k$-times weakly differentiable functions. When $\Gamma_D\subset\partial\Omega$, we denote by $W^{1,p}_D(\Omega)$ the subset of $W^{1,p}(\Omega)$ consisting of functions that vanish on $\Gamma_D$ in the trace sense. This space coincides with the closure of $C^\infty_c(\Omega)$ in $W^{1,p}(\Omega)$ when $\Omega$ is a Lipschitz domain, see the definition below. If $p=2$ we write $H^k(\Omega)$ and $H^1_D(\Omega)$ instead of $W^{k,2}(\Omega)$ and $W^{1,2}_D(\Omega)$. By $C^\alpha(\Omega)$ we denote the space of $\alpha$ H\"older continuous functions for $\alpha\in(0,1)$. The topological dual space of a Banach space $X$ is denoted by $X^*$.

We call a bounded, open set $\Omega\subset\mathbb{R}^d$ a Lipschitz domain if $\overline{\Omega}$ is a Lipschitz manifold with boundary, this definition is adopted from \cite[Definition 1.2.1.2]{grisvard2011elliptic}.
We denote the cube $[-1,1]^n\subset\mathbb{R}^d$ by $Q$, its half $\{x\in Q\mid x_d <0\}$ by $Q_-$, the hyperplane $\{ x\in Q \mid x_d = 0 \}$ by $\Sigma$ and  $\{x\in\Sigma\mid x_{d-1}<0\}$ by $\Sigma_0$. The next definition goes back to Gr\"oger, see \cite{groger1989aw}.
\begin{definition}[Gr\"oger Regular Sets]\label{definition:groger_regular_sets}
    Let $\Omega\subset\mathbb{R}^d$ be bounded and open and $\Gamma\subset\partial\Omega$ a relatively open set. We call $\Omega\cup\Gamma$ Gr\"oger regular, if for every $x\in\partial\Omega$ there are open sets $U,V\subset\mathbb{R}^d$ with $x\in U$, and a bijective, bi-Lipschitz map $\phi:U\to V$, such that $\phi(x) = 0$ and $\phi(U\cap (\Omega\cup\Gamma) )$ is either $Q_-$, $Q_-\cup\Sigma$ or $Q_-\cup\Sigma_0$.
\end{definition}
It can easily be shown that Gr\"oger regular sets $\Omega$ (no matter the choice $\Gamma \subset \partial\Omega$) are Lipschitz domains, we refer to \cite[Theorem 5.1]{haller2009holder}. The notion of Gr\"oger regularity is very weak and many applications fall in this category. This claim is supported by the following characterization of Gr\"oger regular sets in two and three dimensions that allow to check Gr\"oger regularity almost ``by appearance''.
\begin{theorem}[Gr\"oger Regular Sets in 2D, Theorem 5.2 in \cite{haller2009holder}]\label{theorem:groger_regular_sets_in_2d}
    Let $\Omega \subset \mathbb{R}^2$ be a Lipschitz domain and $\Gamma\subset\partial\Omega$ be relatively open. Then $\Omega\cup\Gamma$ is Gr\"oger regular if and only if $\overline{\Gamma}\cap(\partial\Omega\setminus\Gamma)$ is finite and no connected component of $\partial\Omega\setminus\Gamma$ consists of a single point.
\end{theorem}
\begin{theorem}[Gr\"oger Regular Sets in 3D, Theorem 5.4 in \cite{haller2009holder}]\label{theorem:groger_regular_sets_in_3d}
    Let $\Omega \subset \mathbb{R}^3$ be a Lipschitz domain and $\Gamma\subset\partial\Omega$ be relatively open. Then $\Omega\cup\Gamma$ is Gr\"oger regular if and only if the following two conditions hold
    \begin{itemize}
        \item [(i)] $\partial\Omega\setminus\Gamma$ is the closure of its interior.
        \item [(ii)] For any $x\in \overline{\Gamma}\cap(\partial\Omega\setminus\Gamma)$ there is an open neighborhood $U_x$ of $x$ and a bi-Lipschitz map $\phi:U_x\cap\overline{\Gamma}\cap(\partial\Omega\setminus\Gamma)\to (-1,1)$.
    \end{itemize}
\end{theorem}
In the definition of Gr\"oger regular sets, the local model $\{ x\in Q\mid x_d<0 \}\cup \{ x\in Q\mid x_d = 0, \ x_{d-1} <0 \}$ is redundant. We also cite the following.
\begin{lemma}[Lemma 4.10 in \cite{haller2009holder}]
    There exists a bi-Lipschitz mapping $\Psi:\mathbb{R}^d \to \mathbb{R}^d$ mapping $Q_-\cup\Sigma_0$ onto $Q_-\cup\Sigma$.
\end{lemma}

\section{An Elliptic Mixed Boundary Value Problem}
In this section we prove a H\"older regularity result for linear elliptic equations with mixed boundary conditions and measurable, bounded coefficients with explicit control of the H\"older norm in terms of the data. It is the stationary counterpart of Theorem~\ref{theorem:parabolic_quantitative_holder} and of independent interest. The theorem is in the spirit of \cite{stampacchia1960problemi}. However, we extend the results from \cite{stampacchia1960problemi} to Lipschitz domains with a very weak compatibility condition on the Dirichlet-Neumann partition $\partial\Omega = \Gamma_D \cup \Gamma_N$ of the boundary. We follow closely the proof in \cite{haller2009holder} and extend it by explicitly controlling the appearing constants.

\subsection{Elliptic Result}
\begin{theorem}[Quantitative H\"older Control for Mixed Boundary Value Problems]\label{thm:quantitative_holder_regularity}
    Let $\Omega \subset \mathbb{R}^d$ be bounded and open with $d\in \{2,3,4\}$, consider a partition $\partial\Omega = \Gamma_N \cup \Gamma_D$ into Neumann and Dirichlet boundary and assume that $\Omega\cup\Gamma_N$ is Gr\"oger regular. Let $\mathcal{M} \subset L^\infty(\Omega,\mathbb{R}^{d\times d})$ be a set of matrix-valued, measurable functions with a common lower bound $\nu > 0$ on the ellipticity constants and a common upper bound $M$ on the $L^\infty(\Omega,\mathbb R^{d\times d})$ norm. For $A\in\mathcal{M}$ define the operator
    \begin{equation}
        -\operatorname{div}\left(A\nabla\cdot\right) + 1: H^1_D(\Omega) \to H^1_D(\Omega)^*, \quad u\mapsto \int_\Omega A\nabla u\nabla\cdot + u\cdot \mathrm dx.
    \end{equation}
    Then, for every $q>d$ and $A\in\mathcal M$ there exists $\alpha >0$ such that
    \begin{equation*}
        \left( -\operatorname{div} \left(A\nabla\cdot \right) + 1 \right)^{-1}:W^{-1,q}_D(\Omega) \to C^\alpha(\Omega)
    \end{equation*}
    is continuous. Stronger, for all $A\in\mathcal M$ we may choose the same $\alpha >0$ and can estimate the operatornorms
    \begin{equation}
        \sup_{A\in\mathcal M}\left \lVert \left( -\operatorname{div} \left(A\nabla\cdot \right) + 1 \right)^{-1}  \right\rVert_{\mathcal{L}(W^{-1,q}_D(\Omega), C^\alpha(\Omega))} < \infty.
    \end{equation}
\end{theorem}
\begin{proof}
    The idea of the proof is to localize the equation by a partition of unity, additionally employing the Lipschitz transformations from the definition of a Gr\"oger regular set. Using a suitable reflection technique at the Neumann boundary, this allows to apply H\"older regularity results for pure, homogeneous Dirichlet problems either on a ball or a cuboid. In these cases quantitative regularity results exist. The details of the proof are carried out throughout this section. As only the quantitative aspects of the transformations are missing, we pay special attention to these and keep the remaining aspects of the proof brief, referring to \cite{haller2009holder} when necessary.
\end{proof}

\subsection{Known Regularity Results}
We review the known regularity results that we need in the proof of the main theorem. We begin with a classical H\"older regularity result for elliptic equations without mixed boundary conditions.
\begin{theorem}[Theorem C.2 in \cite{kinderlehrer2000introduction}]\label{thm:oscillation_control_pure_dirichlet}
    Let $\Omega \subset \mathbb R^d$ be a ball or a cuboid, $f\in L^q(\Omega,\mathbb{R}^d)$ with $q>d$ and $q>2$. Assume that $A \in L^\infty(\Omega,\mathbb{R}^{d\times d})$ is uniformly elliptic with ellipticity constant $\nu >0$ and $L^\infty(\Omega,\mathbb{R}^{d\times d}$ bound $M >0$. Then, there exist $K=K(\nu,M,\Omega,d) >0$ and $\alpha = \alpha(\nu,M,\Omega,d)\in(0,1)$ such that for the solution $u\in H^1_0(\Omega)$ of 
    \begin{equation*}
        \int_\Omega A\nabla u\nabla(\cdot)\mathrm dx = \int_\Omega f\cdot\nabla(\cdot)\mathrm dx \quad \text{in } H^1_0(\Omega)^*
    \end{equation*}
    it holds $u\in C^0(\Omega)$ and 
    \begin{equation}\label{equation:oscillation_control}
        \max_{\overline{\Omega}\cap B_r(x)}u(x) - \min_{\overline{\Omega}\cap B_r(x)}u(x) = \underset{\overline{\Omega}\cap B_r(x)}{\operatorname{osc}}u \leq K \lVert f \rVert_{L^q(\Omega,\mathbb{R}^d)}\cdot r^\alpha.
    \end{equation}
\end{theorem}
\begin{proof}
    In \cite{kinderlehrer2000introduction}, the result is proven for domains of class $s$ which trivially includes balls and cuboids. For us the result for balls and cuboids suffices.
\end{proof}
The above result implies a control of the H\"older norm. We collect this fact in a Corollary.
\begin{corollary}\label{corollary:holder_control_pure_dirichlet}
    Assume we are in the situation of Theorem \ref{thm:oscillation_control_pure_dirichlet}. Then
    \begin{equation*}
        \left(-\operatorname{div}A\nabla \right)^{-1}:W_0^{1,q'}(\Omega)^* \to C^\alpha(\Omega)
    \end{equation*}
    is well defined and continuous with its operatornorm bounded by
    \begin{equation*}
        \norm{\left(-\operatorname{div}A\nabla \right)^{-1}}_{\mathcal{L}(W_0^{1,q'}(\Omega)^*, C^\alpha(\Omega))} \leq K,
    \end{equation*}
    with $K = K(\nu,M,\Omega,d)$, however, possibly different from the constant $K$ in Theorem \ref{thm:oscillation_control_pure_dirichlet}.
\end{corollary}
\begin{proof}
    We begin by showing that \eqref{equation:oscillation_control} yields a bound on the $C^\alpha(\Omega)$ norm of a solution $u$ to $-\operatorname{div}(A\nabla u) = f$. To this end, take $x,y\in\overline{\Omega}$ and consider the closed ball around $x$ with radius $r=|x-y|$. Then, $y\in B_r(x)$ and \eqref{equation:oscillation_control} yields
    \begin{equation*}
        |u(x)-u(y)| \leq \underset{\overline{\Omega}\cap B_r(x)}{\operatorname{osc}}u \leq K\lVert f \rVert_{L^q(\Omega,\mathbb{R}^d)}|x-y|^\alpha,
    \end{equation*}
    hence
    \begin{equation*}
        |u|_{C^\alpha(\Omega)} \leq K\lVert f \rVert_{L^q(\Omega,\mathbb{R}^d)}.
    \end{equation*}
    To bound the $C^0(\Omega)$ norm of $u$, note that $u$ vanishes on the boundary of $\Omega$. Let $x\in\overline{\Omega}$ and $x_0\in \partial\Omega$ and use again \eqref{equation:oscillation_control} to estimate
    \begin{equation*}
        |u(x)| \leq |u(x)-u(x_0)| \leq K\lVert f \rVert_{L^q(\Omega,\mathbb{R}^d)}|x-x_0|^\alpha \leq K\lVert f \rVert_{L^q(\Omega,\mathbb{R}^d)}\operatorname{diam}(\Omega)^\alpha.
    \end{equation*}
    Hence,
    \begin{equation*}
        \norm{u}_{C^\alpha(\Omega)} \leq K\max(1,\operatorname{diam}(\Omega)^\alpha)\lVert f \rVert_{L^q(\Omega,\mathbb{R}^d)}.
    \end{equation*}
    To conclude the proof, note that any abstract functional $\phi \in W^{1,q'}_0(\Omega)^*$ can be written in the form
    \begin{equation*}
        \phi = \int_\Omega f\cdot\nabla(\cdot)\mathrm dx
    \end{equation*}
    for some $f\in L^q(\Omega,\mathbb{R}^d)$ and clearly it holds for a constant $c=c(\Omega,d)$
    \begin{equation*}
        \norm{\phi}_{W_0^{1,q'}(\Omega)^*} \leq \lVert f \rVert_{L^q(\Omega,\mathbb{R}^d)} \leq c\cdot\norm{\phi}_{W_0^{1,q'}(\Omega)^*}.
    \end{equation*}
    Thus, we can estimate the operatornorm
    \begin{equation*}
        \norm{(-\operatorname{div}(A\nabla))^{-1}}_{\mathcal{L}(W_0^{1,q'}(\Omega)^*),C^\alpha(\Omega)} \leq c\cdot K\max(1,\operatorname{diam}(\Omega)^\alpha)
    \end{equation*}
    as asserted.
\end{proof}
The next result concerns higher integrability of the gradient of the solution of an elliptic equation subjected to mixed boundary conditions. It is essentially to Gr\"oger, see for example \cite{groger1989aw, groger1989resolvent} for the original work and \cite{haller2016elliptic} for a more recent proof that weakens the assumptions on the domain even further. However, we stay in the realm of Gr\"oger regular sets as this seems general enough for the applications we have in mind.
\begin{theorem}[Higher Gradient Integrability, Theorem 5.6 in \cite{haller2016elliptic}]\label{thm:higher_gradient_integrability}
    Let $\mathcal{M}\subset L^\infty(\Omega,\mathbb{R}^{d\times d})$ be a set of matrix valued functions with a common lower bound $\nu >0$ on the ellipticity constants and a common upper bound $M >0$ on the $L^\infty(\Omega,\mathbb{R}^{d\times d})$ norm. Furthermore, assume that $\Omega \cup \Gamma_N$ is Gr\"oger regular. Then, there is an open interval $I_{\mathcal{M}}$ around $2$ such that for all $A\in\mathcal{M}$ and $p\in I_{\mathcal{M}}$
    \begin{equation*}
        -\operatorname{div}(A\nabla) + 1: W^{1,p}_{D}(\Omega) \to W^{1,p'}_D(\Omega)^* 
    \end{equation*}
    is a linear homeomorphism and we have
    \begin{equation*}
        \sup_{p\in I_{\mathcal{M}}}\sup_{A\in\mathcal{M}}\norm{\left( -\operatorname{div}(A\nabla) + 1 \right)^{-1}}_{\mathcal{L}(W^{1,p}_{D}(\Omega), W^{1,p'}_D(\Omega)^*)} < \infty.
    \end{equation*}
\end{theorem}
\begin{proof}
    This is a specialized version of Theorem 5.6 in \cite{haller2016elliptic}. We need to guarantee that our assumptions imply the Assumptions 2.3, 3.1 and 5.4 in the notation of that paper (which they a forteriori do). In fact, Gr\"oger regular sets are Lipschitz domains and this ensures Assumption 2.3 in \cite{haller2016elliptic} and also Assumption 4.11 there. Then, Assumption 4.11 implies Assumption 3.1 as shown in Theorem 4.15 in \cite{haller2016elliptic}. Finally, Assumption 5.4 only requires ellipticity and measurability of the functions $A\in\mathcal{M}$, a fact that we also assumed.
\end{proof}

\subsection{Technical Lemmas}
As the strategy to prove Theorem \ref{thm:quantitative_holder_regularity} consists of localization techniques we investigate in the following technical lemmas how this effects the H\"older control we are interested in. The localization goes through three possible stages: i) a localization by a partition of unity. This involves analyzing how the equation is changed when the solution is multiplied by a smooth cut-off function, ii) in the vicinity of $\partial\Omega$, the Lipschitz transformations to cuboids from the definition of Gr\"oger regular sets need to be employed. This yields a pure Dirichlet problem for the Dirichlet boundary, iii) at the Neumann boundary a reflection technique is used to also produce a pure Dirichlet problem.

The following is a quantitative version of Lemma 4.6 in \cite{haller2009holder}.
\begin{lemma}\label{lemma:localization_cut_off_i}
    Let $\Omega \subset \mathbb R^d$ be open and bounded with a partition $\partial\Omega = \Gamma_D\cup\Gamma_N$ in Dirichlet and Neumann boundary parts. Furthermore, let $\Omega \cup \Gamma_N$ be regular and $\mathcal{U}\subset \mathbb{R}^d$ open, such that $\Omega_\bullet \coloneqq \Omega\cap\mathcal{U}$ is also a Lipschitz domain. Furthermore, set $\Gamma_\bullet\coloneqq \Gamma_D\cap\mathcal{U}$ and let $\eta\in C^\infty_0(\mathbb{R}^d)$ with support in $\mathcal{U}$. For arbitrary but fixed $q\in[1,\infty)$ define the maps
    \begin{enumerate}
        \item [(i)] The multiplication-restriction operator 
        \begin{equation*}
            R_\eta: W^{1,q}_{\Gamma_D}(\Omega) \to W^{1,q}_{\Gamma_\bullet}(\Omega_\bullet), \quad v\mapsto \eta v_{|\Omega_\bullet}.
        \end{equation*}
        \item[(ii)] The multiplication-extension operator
        \begin{equation*}
            E_{\eta}: W^{1,q}_{\Gamma_\bullet}(\Omega_\bullet) \to W^{1,q}_{\Gamma_D}(\Omega), \quad v\mapsto \widetilde{\eta v}.
        \end{equation*}
        Here, the map $v\mapsto\tilde{v}$ denotes the extension by zero outside of $\Omega_\bullet$.
    \end{enumerate}
    Then, both maps are well defined, linear and continuous and we may estimate
    \begin{gather*}
        \lVert \eta v_{|\Omega_\bullet} \rVert_{W^{1,q}_{\Gamma_\bullet}(\Omega_\bullet)} \leq 2 \lVert \eta \rVert_{C^1(\Omega_\bullet)} \norm{v}_{W^{1,q}_{\Gamma_D}(\Omega)} \quad \& \quad \lVert \widetilde{\eta v} \rVert_{W^{1,q}_{\Gamma_D}(\Omega)} \leq 2\lVert \eta\rVert_{C^1(\Omega_\bullet)}\norm{v}_{W^{1,q}_{\Gamma_\bullet}(\Omega_\bullet)}.
    \end{gather*}
\end{lemma}
\begin{proof}
    The well definedness of $R_\eta$ and $E_\eta$ was established in Lemma 4.6 in \cite{haller2009holder}. The estimates can be computed in the following way
    \begin{align*}
        \lVert \eta v_{|\Omega_\bullet} \rVert_{W^{1,q}_{\Gamma_\bullet}(\Omega_\bullet)} 
        &=
        \lVert \eta v \rVert_{L^q(\Omega_\bullet)} + \lVert \nabla(\eta v) \rVert_{L^q(\Omega_\bullet,\mathbb{R}^d)}
        \\
        &\leq 
        \lVert \eta v \rVert_{L^q(\Omega_\bullet)} + \lVert v\nabla\eta \rVert_{L^q(\Omega_\bullet,\mathbb{R}^d)} + \lVert \eta\nabla v \rVert_{L^q(\Omega_\bullet,\mathbb{R}^d)}
        \\
        &\leq
        \lVert \eta \rVert_{C^0(\Omega_\bullet)} \lVert v \rVert_{L^q(\Omega)} + \lVert \nabla \eta \rVert_{C^0(\Omega_\bullet)^d} \lVert v \rVert_{L^q(\Omega)} + \lVert \nabla v \rVert_{L^q(\Omega,\mathbb{R}^d)} \lVert \eta \rVert_{C^0(\Omega_\bullet)}
        \\
        &\leq
        2\lVert \eta \rVert_{C^1(\Omega_\bullet)}\lVert v \rVert_{W^{1,q}(\Omega)}.
    \end{align*}
    The expression $\lVert \widetilde{\eta v} \rVert_{W^{1,q}(\Omega)}$ can be estimated similarly.
\end{proof}

We also need a quantitative version of Lemma 4.7 in \cite{haller2009holder}. 
\begin{lemma}\label{lemma:localization_ii}
    Let $\Omega$, $\Gamma_N$, $\Gamma_D$, $\mathcal{U}$, $\eta$, $\Omega_\bullet$, $\Gamma_\bullet$, $R_\eta$ and $E_\eta$ be as in Lemma \ref{lemma:localization_cut_off_i} and denote by $A_\bullet$ the restriction of a function $A\in L^\infty(\Omega,\mathbb R^{d\times d})$ to the set $\Omega_\bullet$. For $f\in H^1_D(\Omega)^*$ denote by $v_f\in H^1_D(\Omega)$ the function that satisfies
    \begin{equation*}
        -\operatorname{div}\left( A\nabla v_f \right) + v_f = f, \quad \text{in }H^1_{\Gamma_D}(\Omega)^*.
    \end{equation*}
    Define the maps
    \begin{enumerate}
        \item [(i)] The adjoint map of $E_\eta$ for $q\in(1,\infty)$
        \begin{equation*}
            E^*_\eta: W^{1,q'}_{\Gamma_D}(\Omega)^* \to W^{1,q'}_{\Gamma_\bullet}(\Omega_\bullet)^*, \quad f\mapsto f(\widetilde{\eta (\cdot)})\eqqcolon f_\bullet
        \end{equation*}
        \item[(ii)] The functional $T_{v_f}$
        \begin{equation*}
            T_{v_f}: H^1_{\Gamma_\bullet}(\Omega_\bullet) \to \mathbb{R}, \quad w\mapsto \int_{\Omega_\bullet}vA_\bullet \nabla \eta\nabla w\mathrm dx.
        \end{equation*}
    \end{enumerate}
    Then, the localization of $v_f$ by $\eta$, i.e., $u_f\coloneqq (\eta v)_{|\Omega_\bullet}$ satisfies the equation
    \begin{equation}\label{eq:localized_elliptic_equation}
        -\operatorname{div}\left( A_\bullet \nabla u_f \right) = -(\eta v_f)_{|\Omega_\bullet} - (A_\bullet \nabla v_f)_{|\Omega_\bullet}(\nabla \eta)_{|\Omega_\bullet} + T_{v_f} + f_\bullet \eqqcolon f^\bullet \quad \text{in }H^1_{\Gamma_\bullet}(\Omega_\bullet)^*.
    \end{equation}
    Furthermore, if $2\leq d \leq 4$ and $f\in W^{1,q'}_{\Gamma_D}(\Omega)^*$ with $q > d$, then there exists $p > d$ such that $f^\bullet \in W^{1,p'}_{\Gamma_\bullet}(\Omega_\bullet)^*$ and the map
    \begin{equation*}
        \operatorname{Loc}: W^{1,q'}_{\Gamma_D}(\Omega)^* \to W^{1,p'}_{\Gamma_\bullet}(\Omega_\bullet)^*, \quad f\mapsto f^\bullet
    \end{equation*}
    possesses an estimate on its operatornorm only depending on $\nu$, $M$ and $\Omega$, i.e., 
    \begin{equation}\label{eq:norm_control_localized_rhs}
        \lVert f^\bullet \rVert_{W^{1,p'}_{\Gamma_\bullet}(\Omega_\bullet)^*} \leq C(\Omega,\nu,M)\lVert f \rVert_{W^{1,q'}_{\Gamma_D}(\Omega)^*}
    \end{equation}
\end{lemma}
\begin{proof}
    Our extension of Lemma 4.7 in \cite{haller2009holder} is the explicit norm control in \eqref{eq:norm_control_localized_rhs}. To this end, we treat the terms in \eqref{eq:localized_elliptic_equation} separately. First, note that there is $\varepsilon >0$ such that
    \begin{equation*}
        W^{1,4/3-\varepsilon}_{\Gamma_\bullet}(\Omega_\bullet) \hookrightarrow L^{4/3}(\Omega_\bullet)
    \end{equation*}
    and we set $p_1' = 4/3 - \varepsilon$ which implies $p_1 > 4$. We then compute for $w\in W^{1,p_1'}_{\Gamma_\bullet}(\Omega_\bullet)$
    \begin{align*}
        \int_{\Omega_\bullet}\eta v_f w\mathrm dx 
        &\leq
        \lVert \eta \rVert_{L^\infty(\Omega_\bullet)}\lVert v_f \rVert_{L^4(\Omega_\bullet)}\lVert w \rVert_{L^{4/3}(\Omega_\bullet)}
        \\
        &\leq 
        C(\Omega)\lVert \eta \rVert_{L^\infty(\Omega_\bullet)}\lVert v_f \rVert_{H^1_{\Gamma_D}(\Omega)}\lVert w \rVert_{W^{1,p_1'}_{\Gamma_\bullet}(\Omega_\bullet)}
        \\
        &\leq 
        C(\Omega,\nu)\lVert \eta \rVert_{L^\infty(\Omega_\bullet)}\lVert f \rVert_{H^1_{\Gamma_D}(\Omega)^*}\lVert w \rVert_{W^{1,p_1'}_{\Gamma_\bullet}(\Omega_\bullet)}
        \\
        &\leq 
        C(\Omega,\nu)\lVert \eta \rVert_{L^\infty(\Omega_\bullet)}\lVert f \rVert_{W^{1,q'}_{\Gamma_D}(\Omega)^*}\lVert w \rVert_{W^{1,p_1'}_{\Gamma_\bullet}(\Omega_\bullet)}
    \end{align*}
    Taking suprema over unit balls in $W^{1,p_1'}_{\Gamma_\bullet}(\Omega_\bullet)$ and $W^{1,q'}_{\Gamma_D}(\Omega)^*$ we get that the map 
    \begin{equation*}
        W^{1,q'}_{\Gamma_D}(\Omega)^* \to W^{1,p_1'}_{\Gamma_\bullet}(\Omega_\bullet)^*, \quad f \mapsto -\int_{\Omega_\bullet}\eta v_f (\cdot)\mathrm dx
    \end{equation*}
    has its operatornorm bounded by $C(\Omega,\nu)\lVert \eta \rVert_{L^\infty(\Omega_\bullet)}$. 
    
    For the second term, note that we may factorize for all small enough $\varepsilon > 0$ using Theorem \ref{thm:higher_gradient_integrability}
    \begin{equation*}
        W^{1,q'}_{\Gamma_D}(\Omega)^* \hookrightarrow W^{1,(2+\varepsilon)'}_{\Gamma_D}(\Omega)^* \to W^{1,(2+\varepsilon)}_{\Gamma_D}(\Omega) \to L^{2+\varepsilon}(\Omega_\bullet) \hookrightarrow W^{1,p_2'}_{\Gamma_\bullet}(\Omega_\bullet)^*
    \end{equation*}
    given by
    \begin{equation*}
        f\mapsto f \mapsto v_f \mapsto A_\bullet \nabla v_f\nabla\eta_{|\Omega_\bullet} \mapsto \int_{\Omega_\bullet}A_\bullet \nabla v_f\nabla\eta (\cdot)\mathrm dx,
    \end{equation*}
    where $q' \leq (2+\varepsilon)'$ and $1/p_2 \geq (d-2-\varepsilon)/(d(2+\varepsilon))$, meaning $p_2 > 4$, the latter being possible due to $2 \leq d \leq 4$. The latter also implies the continuity of the embedding
    \begin{equation*}
        L^{2+\varepsilon}(\Omega_\bullet) \hookrightarrow W^{1,p_2'}_{\Gamma_\bullet}(\Omega_\bullet)^*.
    \end{equation*}
    The operatornorm of the composition then essentially relies on the operatornorm of 
    \begin{equation*}
        W^{1,(2+\varepsilon)'}_{\Gamma_D}(\Omega)^* \to W^{1,2+\varepsilon}_{\Gamma_D}(\Omega), \quad f\mapsto v_f.
    \end{equation*}
    However, Theorem \ref{thm:higher_gradient_integrability} shows that this is uniform with respect to the ellipticity constant $\nu$ of $A$, its $L^\infty(\Omega,\mathbb{R}^{d\times d})$ bound for $A$ and all small $\varepsilon > 0$.
    
    The third term works similar. Following \cite{haller2009holder} there is $\varepsilon > 0$ such that
    \begin{equation*}
        W^{1,2+\varepsilon}_{\Gamma_D}(\Omega) \hookrightarrow L^{4+\delta}(\Omega)
    \end{equation*}
    for a $\delta = \delta(d) > 0$. We estimate for $w \in W^{1,(4+\delta)'}_{\Gamma_\bullet}(\Omega_\bullet)$
    \begin{align*}
        \langle Tv_f,w \rangle_{W^{1,(4+\delta)'}_{\Gamma_\bullet}(\Omega_\bullet)}
        &\leq
        \lVert v_f \rVert_{L^{4+\delta}(\Omega_\bullet)} \lVert A \rVert_{L^\infty(\Omega,\mathbb{R}^{d\times d})}\lVert \nabla\eta \rVert_{L^\infty(\Omega_\bullet)}\lVert w \rVert_{W^{1,(4+\delta)'}_{\Gamma_\bullet}(\Omega_\bullet)}
        \\
        &\leq
        C(\nu,M,\Omega)\lVert f \rVert_{W^{1,(2+\varepsilon)'}_{\Gamma_D}(\Omega)^*} \lVert A \rVert_{L^\infty(\Omega,\mathbb{R}^{d\times d})}\lVert \nabla\eta \rVert_{L^\infty(\Omega_\bullet)}\lVert w \rVert_{W^{1,(4+\delta)'}_{\Gamma_\bullet}(\Omega_\bullet)}.
    \end{align*}
    The constant $C(\nu,M,\Omega)$ is again determined through Theorem \ref{thm:higher_gradient_integrability}. We set $p_2 = 4+\delta$.
    
    Finally, the mapping $f\mapsto f_\bullet$ is nothing but $E^*_\eta$ and thus $\lVert E^*_\eta\rVert = \lVert E_\eta \rVert$, the latter already being computed in Lemma \ref{lemma:localization_cut_off_i}. To conclude the proof we take $p = \min(p_1,p_2,p_3)$.
\end{proof}

We reproduce the following proposition from \cite{haller2009holder} as the notation it introduces and its content are essential for the remainder of the section.
\begin{proposition}[Proposition 4.9 in \cite{haller2009holder}]\label{proposition:lipschitz_transforms}
    Let $\Omega\subset\mathbb{R}^d$ be a bounded Lipschitz domain, let $\Gamma_N$ be an open subset of its boundary and denote by $\Gamma_D$ its complement in $\partial\Omega$. Let $\phi$ be bi-Lipschitz mapping defined on a neighborhood of $\Omega$ into $\mathbb{R}^d$ and denote $\phi(\Omega) = \widehat{\Omega}$ and $\phi(\Gamma_D) = \widehat{\Gamma}_D$. Then the following holds:
    \begin{enumerate}
        \item [(i)] For any $p\in(1,\infty)$, the mapping $\phi$ induces a linear homeomorphism
        \begin{equation*}
            \Phi_p:W^{1,p}_{\widehat{D}}(\widehat{\Omega}) \to W^{1,p}_{D}(\Omega), \quad u \mapsto u\circ\phi.
        \end{equation*}
        \item[(ii)] If $A$ is a member of $L^\infty(\Omega,\mathbb{R}^{d\times d})$, then
        \begin{equation*}
            -\Phi_{p'}^* \circ \operatorname{div}(A\nabla \Phi_p(\cdot)) = -\operatorname{div}(\widehat{A}\nabla(\cdot))  
        \end{equation*}
        with
        \begin{equation*}
            \widehat{A}(y) = \frac{D\phi(\phi^{-1}(y))}{\det(D\phi)(\phi^{-1}(y))}A(\phi^{-1}(y))(D\phi)^T(\phi^{-1}(y))
        \end{equation*}
        for almost all $y\in\widehat{\Omega}$.
        \item[(iii)] If $A$ is uniformly elliptic and essentially bounded, then so is $\widehat{A}$.
    \end{enumerate}
\end{proposition}
The last result we need is a reflection procedure that allows to transform a mixed Neumann-Dirichlet problem on the model domain $Q_-\cup\Sigma$ to a pure Dirichlet problem on $Q$ and thus makes Corollary \ref{corollary:holder_control_pure_dirichlet} applicable. It is based on Proposition 4.11 in \cite{haller2009holder}.

\begin{lemma}[Reflection Principle]\label{lemma:reflection_principle}
    For $x = (x_1,\dots,x_d)\in\mathbb{R}^d$ we set $x_- = (x_1,\dots,x_{d-1},x_d)$ and for a matrix $A\in\mathbb{R}^{d\times d}$ we define
    \begin{align*}
        A^{-}_{jk} = 
        \begin{cases}
        \; \displaystyle A_{jk} \quad&\text{if }j,k<d, 
        \\ 
        \; -A_{jk} &\text{if }j=d, k\neq d \text{ or } k=d \text{ and }j\neq d, 
        \\
        \; A_{jk} & \text{if }j=k=d.
        \end{cases}
    \end{align*}
    Now let $A$ denote a member of $L^\infty(Q_-,\mathbb{R}^{d\times d})$ and define a member of $L^\infty(Q,\mathbb{R}^{d\times d})$ via
    \begin{align*}
        \hat{A}(x) = 
        \begin{cases}
        A(x) \quad&\text{if }x\in Q_,
        \\
        (A(x_-))^- &\text{if }x_-\in Q_-.
        \end{cases}
    \end{align*}
    Let us set $\Gamma_D = \partial Q_-\setminus \Sigma$. Then for any fixed $p\in(1,\infty)$ it holds:
    \begin{enumerate}
        \item [(i)] If $v\in W^{1,p}_{\Gamma_D}(Q_-)$ satisfies $-\operatorname{div}(A\nabla v) = f\in W^{1,p'}_{\Gamma_D}(Q_-)^*$, then $-\operatorname{div}(\hat A\nabla \hat v ) = \hat f\in W^{1,p'}_0(Q)^*$ holds for
        \begin{align*}
            \hat v(x) = 
            \begin{cases}
                v(x) \quad&\text{if }x\in Q_,
                \\
                v(x_-) &\text{if }x_-\in Q_-
            \end{cases}
        \end{align*}
        and $\langle \hat f, u \rangle_{W^{1,p}_0(Q)} = \langle f, u|_{Q_-} + u_-|_{Q_-} \rangle_{W^{1,p}_{\Gamma_D(Q_-)}}$, where $u_-(x) = u(x_-)$.
        \item [(ii)] The map
        \begin{equation*}
            W^{1,p'}_{\Gamma_D}(Q_-)^* \to W^{1,p'}_0(Q)^*, \quad f\mapsto \hat f
        \end{equation*}
        is continuous.
        \item [(iii)] Furthermore, if $A\in L^\infty(Q_-,\mathbb{R}^{d\times d})$ has ellipticity constant $\nu$ and $L^\infty$ bound $M$, then so does $\hat A$.
    \end{enumerate}
\end{lemma}
\begin{proof}
    The only thing not included in Proposition 4.11 in \cite{haller2009holder} is (iii). However, for all $\xi\in\mathbb R^d$ it holds (as we compute later on)
    \begin{equation*}
        A^-\xi\cdot\xi = A\hat\xi\cdot\hat\xi,
    \end{equation*}
    where $\hat\xi = (-\xi_1,\dots,-\xi_{d-1},\xi_d)$. This implies
    \begin{equation*}
        \inf_{\xi\neq 0}A^-\xi\cdot\xi = \inf_{\xi\neq 0}A\hat\xi\cdot\hat\xi \geq \nu \lvert \hat\xi \rvert^2 = \nu \lvert \xi \rvert^2.
    \end{equation*}
    Furthermore, it holds $\norm{A^-} = \norm{A}$ in the Frobenius norm, hence $\hat A$ and $A$ share its bound as members of $L^\infty(Q_-,\mathbb R^{d\times d})$. Finally, we provide the computations for the above equality
    \begin{align*}
        A^-\xi\cdot\xi &= \sum_{i,j=1}^{d-1}A_{ij}\xi_j\xi_i + \sum_{i=1}^{d-1}(-A_{id})\xi_d\xi_i + \sum_{j=1}^{d-1}(-A_{dj})\xi_j\xi_d + A_{dd}\xi_d^2
        \\
        &=
        \sum_{i,j=1}^{d-1}A_{ij}(-\xi_j)(-\xi_i) + \sum_{i=1}^{d-1}A_{id}\xi_d(-\xi_i) + \sum_{j=1}^{d-1}A_{dj}(-\xi_j)\xi_d + A_{dd}\xi_d^2
        \\
        &=
        A\hat\xi\cdot\hat\xi.
    \end{align*}
\end{proof}

\subsection{Proof of the Main Result}
\begin{proof}[Proof of Theorem \ref{thm:quantitative_holder_regularity}]
    We follow the steps in \cite{haller2009holder}. For every $x \in \Omega$ choose a ball $B_x \subset \Omega$ centered at $x$ and contained in $\Omega$. For every $x\in \partial\Omega$, by the definition of Gr\"oger regularity, there exists an open neighborhood $U_x$ of $x$ and an open set $W_x$ together with a bi-Lipschitz map $\Psi_x:U_x\to W_x$ such that
    \begin{equation*}
        \Psi_x((\Omega\cup\Gamma_N)\cap U_x) = Q_-\quad \text{or} \quad \Psi_x((\Omega\cup\Gamma_N)\cap U_x) = Q_-\cup \Sigma
    \end{equation*}
    depending on $x\in\partial\Omega$. The system $\{ U_x \}_{x\in\partial\Omega}\cup\{B_x\}_{x\in\Omega}$ forms an open covering of $\overline{\Omega}$. We choose a finite subcovering $U_{x_1},\dots,U_{x_k},B_{x_1,\dots,B_{x_l}}$ and a subordinated smooth partition of unity $\eta_1,\dots,\eta_k,\zeta_1,\dots\zeta_l$. Let $A\in \mathcal{M}$, $q>d$ and $f\in W^{1,q'}_{\Gamma_D}(\Omega)^*$ and denote by $v$ the solution of 
    \begin{equation*}
        -\operatorname{div}(A\nabla v) + v = f, \quad \text{in }H^1_{\Gamma_D}(\Omega)^*.
    \end{equation*}
    Then we use the partition of unity to write
    \begin{equation*}
        v = \sum_{i=1}^k\eta_iv + \sum_{j=1}^l\zeta_jv
    \end{equation*}
    and we need to estimate $\lVert \eta_i v \rVert_{C^\alpha(\Omega)}$ and $\lVert \zeta_j v \rVert_{C^\alpha(\Omega)}$. This leads to three cases that need to be treated differently: First, the $\zeta_j v$ on the balls $B_{x_j}$, then $\eta_i v$ when $(\Omega\cup\Gamma_N)\cap U_x$ equals $Q_-$ and finally the case when $(\Omega\cup\Gamma_N)\cap U_x = Q_-\cup\Sigma$.
    
    \textbf{First Case.} We show that the H\"older norm of the $\zeta_jv$ can be controlled in terms of $C(B_{x_j},\nu,M)\lVert f \rVert_{W^{1,q'}_{\Gamma_D}(\Omega)^*}$. To this end, we employ Lemma \ref{lemma:localization_ii} with $\mathcal{U} = B_{x_j}$, hence $\Omega_\bullet = B_{x_j}$ and $\Gamma_\bullet = \emptyset$. Then $\zeta_jv_{|B_{x_j}}$ satisfies an equation of the form
    \begin{equation*}
        -\operatorname{div}(A_\bullet \nabla(\zeta_jv_{|B_{x_j}})) = g_j \quad \text{in }W^{1,p_j'}_0(B_{x_j})
    \end{equation*}
    with $p_j > d$ and it holds
    \begin{equation*}
        \lVert g_j \rVert_{W^{1,p_j'}_0(B_{x_j})} \leq C(B_{x_j},\nu,M) \cdot \lVert f \rVert_{W^{1,q'}_{\Gamma_D}(\Omega)^*}.
    \end{equation*}
    Hence, by Corollary \ref{corollary:holder_control_pure_dirichlet}, there is $\alpha_j\in(0,1)$ such that
    \begin{equation*}
        \lVert \zeta_j v \rVert_{C^{\alpha_j}(\Omega)} = \lVert \zeta_j v_{|B_{x_j}} \rVert_{C^{\alpha_j}(B_{x_j})} 
        \leq C(B_{x_j},\nu,M) \cdot \lVert f \rVert_{W^{1,p_j'}_{0}(B_{x_j})^*}
        \leq C(B_{x_j},\nu,M) \cdot \lVert f \rVert_{W^{1,q'}_{\Gamma_D}(\Omega)^*}.
    \end{equation*}
    
    \textbf{Second Case.} Here we assume that we use $\eta_j$ subordinated to $U_j$ with
    \begin{equation}\label{case_q_minus}
        \Psi_{x_j}((\Omega\cup\Gamma_N)\cap U_{x_j}) = Q_-.
    \end{equation}
    Setting $\Omega_j = \Omega\cap U_{x_j}$, Lemma \ref{lemma:localization_cut_off_i} shows that $\eta_j v_{|\Omega_j}$ is a member of $H^1_0(\Omega_j)$ and Lemma \ref{lemma:localization_ii} implies that $\eta_j v_{|\Omega_j}$ solves
    \begin{equation*}
        -\operatorname{div}(A_\bullet\nabla(\eta_j v_{|\Omega_j})) = f_j, \quad \text{in }H^1_{0}(\Omega_j)^*
    \end{equation*}
    with $f_j \in W^{1,p_j'}_{0}(\Omega_j)^*$ and $p_j > d$ and again
    \begin{equation*}
        \lVert f_j \rVert_{W^{1,p_j'}_{0}(\Omega_j)^*} \leq C(\Omega_j, \nu, M)\cdot \lVert f \rVert_{W^{1,q'}_{\Gamma_D}(\Omega)^*}.
    \end{equation*}
    Now, transform the function to $Q_-$ using Proposition \ref{proposition:lipschitz_transforms} with $\phi = \Psi^{-1}_{x_j}$ setting
    \begin{equation*}
        \psi_j \coloneqq \Phi_{p_j}(\eta_j v_{|\Omega_j}) = (\eta_j v_{|\Omega_j}) \circ \Psi_{x_j}^{-1}.
    \end{equation*}
    As we assumed \eqref{case_q_minus}, $\eta_j v_{|\Omega_j}$ is a member of $H^1_0(\Omega_j)$ and $\psi_j$ is a member of $H^1_0(Q_-)$. Furthermore, $\psi_j$ satisfies and equation of the form
    \begin{equation*}
        -\operatorname{div}(\tilde A \nabla \psi_j) = h_j\coloneqq (\Phi_{p_j}^*)^{-1}f_j \quad \text{in }W^{1,p_j'}_0(Q_-)^*
    \end{equation*}
    and by Corollary \ref{corollary:holder_control_pure_dirichlet} there is $\alpha_j\in(0,1)$ such that $\psi_j \in C^{\alpha_j}(Q_-)$ with
    \begin{equation*}
        \lVert \psi_j \rVert_{C^{\alpha_j}(Q_-)} \leq C(\nu,M,Q_-)\cdot\lVert h_j \rVert_{W^{1,p_j'}(Q_-)^*},
    \end{equation*}
    where we used that $\tilde A$ is still a bounded, measurable, elliptic matrix with possibly different boundedness and ellipticity constants, however controlled through the geometry of $\Omega_j$. As Lipschitz maps preserve H\"older continuity in a controlled way we also have
    \begin{equation*}
        \lVert \eta_j v_{|\Omega_j} \rVert_{C^{\alpha_j}(\Omega_j)} \leq C(\Omega_j)\cdot\lVert \psi_j \rVert_{C^{\alpha_j}(Q_-)}.
    \end{equation*}
    Finally, we may estimate
    \begin{align*}
        \lVert \eta_j v \rVert_{C^{\alpha_j}(\Omega)}
        =
        \lVert \eta_jv_{|\Omega_j} \rVert_{C^{\alpha_j}(\Omega_j)}
        \leq
        C(\Omega_j)\cdot \lVert \psi_j \rVert_{C^{\alpha_j}(Q_-)}
        &\leq
        C(\nu,M,\Omega_j)\cdot \lVert (\Phi_{p'_j}^*)^{-1}f_j \rVert_{W^{1,p_j'}_0(Q_-)^*}
        \\
        &\leq
        C(\nu,M,\Omega_j)\cdot \lVert f_j \rVert_{W^{1,p_j'}_{0}(\Omega_j)^*}
        \\
        &\leq
        C(\nu,M,\Omega_j)\cdot\lVert f \rVert_{W^{1,q'}_{\Gamma_D}(\Omega)^*}.
    \end{align*}
    
    \textbf{Third Case.} We use the same notation as in the second case but now it holds
    \begin{equation*}
        \Psi_{x_j}((\Omega\cup\Gamma_N)\cap U_{x_j}) = Q_-\cup\Sigma.
    \end{equation*}
    Setting $\Gamma_j = \partial\Omega_j\setminus\Gamma_N$, it holds again $-\operatorname{div}(A_\bullet \nabla (\eta_jv_{|\Omega_j})) = f_j$ in $H^1_{\Gamma_j}(\Omega_j)^*$ with $f_j\in W^{1,p_j'}_{\Gamma_j}(\Omega_j)^*$ and $p_j > d$ and an estimate of the form
    \begin{equation*}
        \lVert f_j \rVert_{W^{1,p_j'}_{\Gamma_j}(\Omega_j)^*} \leq C(\Omega_j,\nu,M)\cdot\lVert f \rVert_{W^{1,q'}_{\Gamma_D}(\Omega)^*}.
    \end{equation*}
    Now we transform to $Q_-$ as in the second case and then use the reflection principle, see Lemma \ref{lemma:reflection_principle} to transform to $Q$. This yields $\psi_j$ and $\hat\psi_j$, the latter solving a homogeneous problem on $Q$, the former as above, however with a Neumann condition on $\Sigma$. We may estimate for a suitable $\alpha_j\in(0,1)$
    \begin{align*}
        \lVert \eta_jv \rVert_{C^{\alpha_j}(\Omega)} 
        =
        \lVert \eta_j v \rVert_{C^{\alpha_j}(\Omega_j)}
        \leq 
        C(\Omega_j)\cdot \lVert \psi_j \rVert_{C^{\alpha_j}(Q_-)}
        &\leq
        C(\Omega_j)\cdot \lVert \hat\psi_j \rVert_{C^{\alpha_j}(Q)}
        \\&\leq
        C(\nu,M,\Omega_j)\cdot\lVert \hat h_j \rVert_{W^{1,p_j'}_0(Q)^*}
        \\&\leq
        C(\nu,M\Omega_j)\cdot\lVert (\Phi^*_{p_j'})^{-1}f_j \rVert_{W^{1,p_j'}_{\partial\Omega\setminus\Sigma}(Q_-)^*}
        \\&\leq
        C(\nu,M,\Omega_j)\cdot\lVert f_j \rVert_{W^{1,p_j'}_{\Gamma_j}(\Omega_j)^*}
        \\&\leq
        C(\nu,M,\Omega_j)\cdot\lVert f \rVert_{W^{1,q'}_{\Gamma_D}(\Omega)^*}.
    \end{align*}
    Taking the minimal $\alpha_j$ concludes the proof.
\end{proof}

\section{Parabolic H\"older Regularity}

\setcounter{definition}{0}

In this section, we prove Theorem~\ref{theorem:parabolic_quantitative_holder} which we restate here for the readers convenience.
\begin{theorem}\label{thm:main_repeat}
    Let $\Omega \subset \mathbb{R}^d$ with $d=2,3$ be a Lipschitz domain, $I=[0,T]$ a time interval, $\partial\Omega = \Gamma_N\cup\Gamma_D$ a partition of the boundary into a Dirichlet and a Neumann part, where both $\Gamma_N$ and $\Gamma_D$ are allowed to have vanishing measure. Assume that $\Omega\cup\Gamma_N$ is Gr\"oger regular, let $f\in L^p(I,L^2(\Omega))$ for $p\in[2,\infty)$, $D\in L^\infty(\Omega,\mathcal{M}_s)$ with ellipticity constant $\nu > 0$ and let $k > 0$ be a constant. For $v_0\in L^\infty(\Omega)$ denote by $v\in H^1(I,H^1_D(\Omega),H^1_D(\Omega)^*)$ the solution to 
    \begin{align*}
        \int_I \langle d_tv,\cdot\rangle_{H^1_D(\Omega)}\mathrm dt + \int_I\int_\Omega D\nabla v \nabla \cdot + kv(\cdot)\mathrm dx\mathrm dt
        &=
        \int_I\int_\Omega f(\cdot)\mathrm dx\mathrm dt \quad \text{in }L^2(I,H^1_D(\Omega))^*
        \\
        v(0) &= v_0.
    \end{align*}
    Then there is $\beta = \beta(p)\in (0,1)$ such that $v\in L^p(I,C^\beta(\Omega))$ and we may estimate
    \begin{equation*}
        \norm{v}_{L^p(I,C^\beta(\Omega))} \leq C\left(\Omega, T, \nu, \lVert D \rVert_{L^\infty(\Omega,\mathbb{R}^{d\times d})}, p, \beta \right)\cdot \left[ \lVert f \rVert_{L^p(I,L^2(\Omega))} + \lVert v_0 \rVert_{L^\infty(\Omega)} \right].
     \end{equation*}
     In the above estimate, if we fix $\Omega$ and $p$, only a lower bound for $\nu$ and upper bounds for $\norm{D}$ and $T$ determine the value of the constant $C$. This provides uniformity for $\nu\in[c_E,C_E]$, $D\in L^\infty(\Omega, \mathcal{M}_s)$ with $\norm{D}\leq C_B$ and time intervals $I^*=[0,T^*]$ with $T^* \leq T$. 
\end{theorem}

\setcounter{definition}{13}

\begin{proof}[Strategy of the Proof]   
    Here we discuss only the main ideas and provide the details in the course of the section. The first ingredient in the proof is the $C^\alpha(\Omega)$ regularity result for the stationary operator, see Theorem~\ref{thm:quantitative_holder_regularity}. This opens the door for maximal parabolic regularity results, however, the initial value as a member of $L^\infty(\Omega)$ does not suffice for a direct application of the theory, which would require $v_0$ to be a member of $H^1_D(\Omega)$, the trace space in this situation, compare to \cite{arendt2017jl}. Therefore, we propose to use the superposition principle for linear operators to split the equation into
    \begin{align*}
        d_tv_1 + \mathcal{M}v_1 &= f,
        \\
        v_1(0) &= 0
    \end{align*}
    and 
    \begin{align*}
        d_tv_2 + \mathcal{M}v_2 &= 0,
        \\
        v_2(0) &= v_0.
    \end{align*}
    The linearity of the equation implies that $v = v_1 + v_2$. This gives us the advantage to analyze $v_1$ and $v_2$ separately. Now, $v_1$ can be treated by a combination of the maximal regularity results in \cite{amann1995linear} and Theorem~\ref{thm:quantitative_holder_regularity}. For $v_2$ we can quantify the norm blow-up at the initial time-point using standard results from \cite{brezis2010functional}. More precisely, it holds
    \begin{equation*}
        \norm{v_2(t)}_{C^\alpha(\Omega)} \leq C\cdot\left(\frac1t \norm{v_0}_{L^2(\Omega)} + 1\right)
    \end{equation*}
    and using an interpolation result we are able to mitigate the singularity of $t \mapsto t^{-1}$ by reducing the H\"older exponent.
\end{proof}

\subsection{Proof of the Main Result} 
We need some basic facts from semi-group theory for linear, unbounded operators in a Hilbert space $H$, that is operators of the form $M:\operatorname{dom}(M)\subset H \to H$. However, we started with a linear, bounded and coercive operator defined on a full space $X$ taking values in its dual, i.e., $\mathcal{M}\in\mathcal{L}(X,X^*)$. If we are given a Gelfand triple structure $(i,X,H)$, that is $X$ and $H$ are Hilbert spaces and $i:X\to H$ is an embedding with dense image, i.e., linear, continuous and bounded, we see that the two concepts are closely related.
\begin{definition}
    Let $(i,X,H)$ be a Gelfand triple and $\mathcal{M}\in\mathcal{L}(X,X^*)$ a coercive bounded linear operator. We define its \emph{part in} $H$ as follows
    \begin{equation*}
        \operatorname{dom}(M) \coloneqq \left\{ v \in X \mid \text{there is }f\in H \text{ with }(f,\cdot)_H = \mathcal{M}v \right\}
    \end{equation*}
    and
    \begin{equation*}
        M:\operatorname{dom}(M)\subset H \to H, \quad Mv = R^{-1}\left(\mathcal{M}v\right)
    \end{equation*}
    where $R$ denotes the Riesz isometry of $H$.
\end{definition}
\begin{remark}
    Note that the above definition suppresses the embedding $i$ in various places, treating it like a set-theoretic inclusion. Furthermore, we stress that $M$ is well defined as a map since for every $\mathcal{M}v$ there is at most one $f\in H$ satisfying $(f,\cdot)_H = \mathcal{M}v$ as $i(X)$ is dense in $H$ by assumption.  
\end{remark}
\begin{lemma}\label{lemma:operator_to_semigroup}
    Let $(i,X,H)$ be a Gelfand triple and $\mathcal{M}\in\mathcal{L}(X,X^*)$ a coercive, bounded linear operator. Then, its part $M$ in $H$ is maximal monotone, thus densely defined. If $\mathcal{M}$ is self-adjoint\footnote{We call a map $T\in\mathcal{L}(X,X^*)$ self-adjoint if $T^*\circ J = T$, where $J:X\to X^{**}$ is the natural isometric embedding of a Banach space into its bi-dual and $T^*$ denotes the usual adjoint map.} as a member of $\mathcal{L}(X,X^*)$, then $M$ is self-adjoint as a densely defined operator in $H$.
\end{lemma}
\begin{proof}
    Let $u,v\in\operatorname{dom}(M)$ and note that by the definition of $\mathcal{M}$ it holds
    \begin{equation}\label{local_equation_m_mathcal_m}
        (Mu,v)_H = (R^{-1}(\mathcal{M}(u)),v)_H = \langle \mathcal{M}u,v \rangle_X.
    \end{equation}
    This identity makes clear that the coercivity of $\mathcal{M}$ implies the monotonicity of $M$. Additionally,
    \begin{equation*}
        \operatorname{Id}|_H + M:\operatorname{dom}(M) \to H
    \end{equation*}
    is bijective and hence $M$ is maximal monotone. If $\mathcal{M}$ is self-adjoint, then \eqref{local_equation_m_mathcal_m} shows that $M$ is symmetric. However, linear symmetric maximal monotone operators are self-adjoint, see \cite{brezis2010functional}.
\end{proof}

The following Proposition is tailored to allow the application of Hille-Yosida's celebrated theorem on solutions to the Cauchy problem.
\begin{proposition}\label{proposition:hille_yoside_prerequisits}
    Let $\Omega\subset\mathbb{R}^d$, $d=1,2,3$ be a bounded domain with a partition of the boundary into Dirichlet and Neumann part $\partial\Omega = \Gamma_N \cup \Gamma_D$. Both $\Gamma_D$ and $\Gamma_N$ are allowed to have vanishing measure. We assume that $\Omega\cup\Gamma_N$ is Gr\"oger regular. Further, let $D\in L^\infty(\Omega, \mathcal{M}_s)$ be given and assume it is elliptic with ellipticity constant $\nu > 0$. Let $k >0$, we define the operator
    \begin{equation*}
        \mathcal{M}: H^1_D(\Omega)\to H^1_D(\Omega)^*, \quad \mathcal{M}v = \int_\Omega D\nabla v\nabla\cdot + kv(\cdot)\mathrm dx.
    \end{equation*}
    Then its part in $L^2(\Omega)$ is maximal monotone and self-adjoint. Further, there exists $\alpha > 0$ such that we have the embedding 
    \begin{equation*}
        \left( \operatorname{dom}(M), \norm{\cdot}_{L^2(\Omega)} + \lVert\cdot \rVert_{L^2(\Omega)} \right) \hookrightarrow C^\alpha(\Omega)
    \end{equation*}
    together with the estimate
    \begin{equation*}
        \norm{u}_{C^\alpha(\Omega)} \leq C(\Omega,\nu,\norm{D}_{L^\infty(\Omega,\mathcal{M}_s)})\cdot \norm{u}_{\operatorname{dom}(M)}.
    \end{equation*}
    Here, the constant $C$ is precisely $\lVert \mathcal{M}^{-1} \rVert_{\mathcal{L}(L^2(\Omega),C^\alpha(\Omega))}$ and depends only on a lower bound for the ellipticity constant and an upper bound on $\norm{D}_{L^\infty(\Omega,\mathcal{M}_s)}$.
\end{proposition}
\begin{proof}
    Using the Gelfand triple $(\operatorname{Id}_{|L^2(\Omega)}, H^1_D(\Omega),L^2(\Omega))$, we can apply Lemma~\ref{lemma:operator_to_semigroup} and deduce the maximal monotonicity of $M$. Further, the symmetry assumption on $D$ implies that $M$ is self-adjoint, again through Lemma~\ref{lemma:operator_to_semigroup}. It remains to show the embedding into H\"older spaces -- essentially due to Theorem~\ref{thm:quantitative_holder_regularity} -- which yields the existence of $\alpha >0$ such that
    \begin{equation*}
        M^{-1}:L^2(\Omega) \to C^\alpha(\Omega)
    \end{equation*}
    is well defined and continuous. This requires the assumption $d=1,2,3$. To see that the graph norm on $\operatorname{dom}(M)$ controls the $\alpha$-H\"older norm, we let $u\in \operatorname{dom}(M)\subset C^\alpha(\Omega)$. Then there exists a unique $f\in L^2(\Omega)$ such that $u = M^{-1}f$ and we compute
\begin{equation*}
    \lVert u \Vert_{C^\alpha(\Omega)} = \lVert M^{-1}f \Vert_{C^\alpha(\Omega)} \leq C \lVert f \Vert_{L^2(\Omega)} = C \lVert Mu \Vert_{L^2(\Omega)} \leq C \lVert u \Vert_{\operatorname{dom}(M)}.
\end{equation*}
The only appearing constant is the operator norm of $M^{-1}$ and Theorem~\ref{thm:quantitative_holder_regularity} guarantees a suitable bound of this norm.
\end{proof}

\begin{theorem}\label{theorem:brezis_estimate} Assume we are in the situation of Proposition~\ref{proposition:hille_yoside_prerequisits}. Then for every $v_0\in L^2(\Omega)$ there exists $\alpha > 0$ and 
\begin{equation*}
    v \in C^1((0,T], L^2(\Omega)) \cap C^0((0,T],C^\alpha(\Omega))
\end{equation*}
solving
\begin{align}\label{equation:cauchy_problem}
    v'(t) + Mv(t) &= 0 \quad \text{on }(0,T]
    \\
    v(0) &= v_0 \notag
\end{align}
Furthermore, it holds
\begin{equation*}
    \lVert v(t) \Vert_{C^\alpha(\Omega)} \leq C\left(\Omega, \nu, \lVert D \rVert_{L^\infty} \right)\left( 1 + \frac{1}{t} \right)\lVert v_0 \Vert_{L^2(\Omega)}.
\end{equation*}
More precisely, the constant $C\left(\Omega, \lfloor D \rfloor, \lVert D \rVert_{L^\infty} \right)$ is the operatornorm of the embedding $\operatorname{dom}(M)\hookrightarrow C^\alpha(\Omega)$.
\end{theorem}
\begin{proof}
    From Theorem 7.7 in \cite{brezis2010functional} it follows that
    \begin{equation*}
        \lVert Mv(t) \Vert_{L^2(\Omega)} \leq \frac{1}{t}\lVert v_0 \Vert_{L^2(\Omega)} \quad \text{and} \quad \lVert v(t) \Vert_{L^2(\Omega)} \leq \lVert v_0 \Vert_{L^2(\Omega)}.
    \end{equation*}
    Using this and the embedding $\operatorname{dom}(M)\hookrightarrow C^\alpha(\Omega)$, we get
    \begin{align*}
        \lVert v(t) \Vert_{C^\alpha(\Omega)} \leq C \lVert v(t) \Vert_{\operatorname{dom}(M)} &= C\lVert v(t) \Vert_{L^2(\Omega)} + C\lVert Mv(t) \Vert_{L^2(\Omega)} 
        \\&
        \leq C\lVert v_0 \Vert_{L^2(\Omega)} + \frac{C}{t}\lVert v_0 \Vert_{L^2(\Omega)}.
    \end{align*}
\end{proof}

\begin{theorem}\label{thm:for_v_2}
    Assume we are in the situation of Proposition~\ref{proposition:hille_yoside_prerequisits} and assume that $v_0\in L^\infty(\Omega)$ and denote by $v\in C^1((0,T],L^2(\Omega))$ the solution to \eqref{equation:cauchy_problem}. Then for every $q\in(1,\infty)$ there exists $\beta = \beta(q)$ such that $v$ is a member of $L^q(I,C^\beta(\Omega)) \cap L^\infty(I,C^0(\Omega))$. Furthermore, we can bound the $L^q(I,C^\beta(\Omega))$ norm depending on the data of the problem in the following way
    \begin{equation}\label{local_equation_estimate}
        \lVert v \rVert_{L^q(I,C^\beta(\Omega))} \leq C\left(\Omega, \nu, \lVert D \rVert_{L^\infty}, \lVert v_0 \rVert_{L^\infty}, I, \alpha, q \right).
    \end{equation}
\end{theorem}
\begin{proof}
    Let $p>q$ be fixed. Choose $\beta > 0$ such that $\alpha/p > \beta$. Then we can estimate for every $u \in C^\alpha(\Omega)$
    \begin{equation*}
        \lVert u \Vert_{C^\beta(\Omega)} \leq C\cdot\lVert u \Vert_{C^0(\Omega)}\lVert u \Vert_{C^\alpha(\Omega)}^{1/p} + \lVert u  \Vert_{C^0(\Omega)}.
    \end{equation*}
    To see this compute
    \begin{align*}
        [u]_\beta &= \sup_{x\neq y}\frac{|u(x) - u(y)|^{1 - 1/p} |u(x) - u(y)|^{1/p} }{|x - y|^{\alpha/p + (\beta - \alpha/p)}}
        \\&=
        \sup_{x\neq y}|u(x) - u(y)|^{1 - 1/p}|x - y|^{\alpha/p - \beta} \cdot \left[ \frac{|u(x) - u(y)|}{|x - y|^\alpha} \right]^{1/p} 
        \\&\leq
        \left( 2\lVert u \Vert_{C^0(\Omega)} \right)^{1 - 1/p}\operatorname{diam}(\Omega)^{\alpha/p - \beta}[u]^{1/p}_\alpha. 
    \end{align*}
    Using the following estimate
    \begin{equation*}
        \norm{v(t)}_{C^0(\Omega)} \leq \norm{v_0}_{L^\infty(\Omega)}
    \end{equation*}
    and the above estimates of the $C^\beta$ norm and Theorem~\ref{theorem:brezis_estimate} we obtain
    \begin{align*}
        \lVert v(t) \Vert_{C^\beta(\Omega)} 
        &\leq 
        \left( 2\lVert v(t) \rVert_{C^0(\Omega)} \right)^{1-1/p}\operatorname{diam}(\Omega)^{\alpha/p - \beta} [v(t)]_\alpha^{1/p} + \lVert v(t) \Vert_{C^0(\Omega)}
        \\&\leq 
        \max\left( 1, 2\lVert v_0 \rVert_{L^\infty(\Omega)} \right) \cdot \max\left( 1, \operatorname{diam}(\Omega) \right) \cdot [v(t)]_\alpha^{1/p} + \lVert v_0 \rVert_{L^\infty(\Omega)}
        \\&\leq
        C\left( \lVert v_0 \rVert_{L^\infty(\Omega)}, \Omega \right) \cdot [v(t)]_\alpha^{1/p} + \lVert v_0 \rVert_{L^\infty(\Omega)}
        \\&\leq
        C\left( \lVert v_0 \rVert_{L^\infty(\Omega)}, \nu, \lVert D \rVert_{L^\infty}, \Omega \right) \cdot \left( 1 + \frac1t \right)^{\frac1p}
    \end{align*}
    Inferring $q/p < 1$ then shows the integrability of $\lVert v(t) \Vert_{C^\beta(\Omega)}^q$ and the asserted bound.
\end{proof}
\begin{remark}
    The constant in \eqref{local_equation_estimate} only depends on the length of the interval $I$, a lower bound for $\nu$ and an upper bound for $\norm{D}_{L^\infty(\Omega)}$, hence is uniform for suitable families of operators and time intervals.
\end{remark}

Finally we cite a known result from \cite{amann1995linear} to treat the case with the vanishing initial condition. 
\begin{theorem}\label{thm:amann}
    Assume we are in the situation of Proposition~\ref{proposition:hille_yoside_prerequisits}. Let $f\in L^p(I,L^2(\Omega))$ with $p\in[2,\infty)$ and denote by $u$ the solution to 
    \begin{align*}
    u'(t) + Mu(t) &= f \quad \text{on }(0,T]
    \\
    u(0) &= 0. \notag
\end{align*}
Then it holds $u\in W^{1,p}(I,L^2(\Omega))\cap L^p(I,\operatorname{dom}(M))$ with the estimate
\begin{equation*}
    \norm{u}_{W^{1,p}(I,L^2(\Omega))\cap L^p(I,\operatorname{dom}(M))} \leq C(\nu,\lVert D \rVert_{L^\infty(\Omega)}, p, I)\cdot \norm{f}_{L^p(I,L^2(\Omega))}
\end{equation*}
where $C(\nu,\lVert D \rVert_{L^\infty(\Omega)}, p, I)$ does depend on a lower bound for $\nu$, on an upper bound for $\lVert D \rVert_{L^\infty}$ and the upper bound $T$ of the time interval $I = [0,T]$.
\end{theorem}
\begin{proof}
    We apply Theorem 4.10.8 in \cite{amann1995linear}, using $E_0 = L^2(\Omega)$, $E_1 = \operatorname{dom}(M)$. The requirement of $E_0$ being an UMD space holds as it is a Hilbert space, the other requirements can be shown using the fact that $M$ is self-adjoint and coercive, i.e., a member of $\mathcal{BIP}(L^2(\Omega);1,0)$ in the terminology of \cite{amann1995linear}. As we consider a problem with homogeneous initial conditions we don't need to concern ourselves with the trace space for the initial conditions.
\end{proof}
\begin{proof}[Completion of the proof of Theorem~\ref{thm:main_repeat}]
    Employing Theorem~\ref{thm:for_v_2} for $v_2$ and Theorem~\ref{thm:amann} for $v_1$ we conclude that
    \begin{equation*}
        \norm{v}_{L^p(I,C^\alpha(\Omega))} 
        \leq
        \norm{v_1}_{L^p(I,C^\alpha(\Omega))} + \norm{v_2}_{L^p(I,C^\alpha(\Omega))}
        \leq 
        C\left(\nu,\lVert D \rVert_{L^\infty}, p, I, \beta \right)\cdot\left( \lVert f\rVert_{L^p(I,L^2(\Omega))} + \lVert v_0 \rVert_{L^\infty(\Omega)} \right).
    \end{equation*}
\end{proof}

\bibliographystyle{apalike}
\bibliography{references}

\end{document}